\documentclass[a4paper, 10pt]{article}

\usepackage{amsmath}
\usepackage{hyperref}
\usepackage{mathtools}
\usepackage{amsfonts}
\usepackage{amssymb}
\usepackage{amsthm}
\usepackage{mathrsfs}
\newtheorem{question}{Question}[section]
\newtheorem{theorem}{Theorem}[section]

\newtheorem{lemma}[theorem]{Lemma}

\newtheorem{corollary}[theorem]{Corollary}
\theoremstyle{definition}
\newtheorem{definition}[theorem]{Definition}
\newtheorem{example}[theorem]{Example}

\newtheorem{remark}[theorem]{Remark}

\newcommand{\norm}[1]{\left\lVert#1\right\rVert}
\usepackage{titling}
\newcommand{\subtitle}[1]{%
  \posttitle{%
    \par\end{center}
    \begin{center}\large#1\end{center}
    \vskip0.5em}%
}
\newcommand{\property}[1]{\textbf{(#1)}}
\title{Distribution of matrices over $\mathbb{F}_q[x]$}
\author{Yibo Ji}
\date{}

\begin{document}

\maketitle
\begin{abstract} 
  In this paper, we count the number of matrices $A = (A_{i,j} )\in \mathcal{O} \subset Mat_{n\times n}(\mathbb{F}_q[x])$ where $\deg(A_{i,j})\leq k, 1\leq i,j\leq n$, $\deg(\det A) = t$, and $\mathcal{O}$ is a given orbit of $GL_n(\mathbb{F}_q[x])$. By an elementary argument, we show that the above number is exactly $\# GL_n(\mathbb{F}_q)\cdot q^{(n-1)(nk-t)}$. This formula gives an equidistribution result over $\mathbb{F}_q[x]$, which is an analogue, in strong form, of a result over $\mathbb{Z}$ proved in \cite{DRS93} and \cite{EM93}.
\end{abstract}

\section{Background}
Two papers from 1993, \cite{EM93} and \cite{DRS93},
 considered the problem of integer point-counting in the setting of algebraic groups. They linked this problem to volumes on quotient spaces over affine symmetric varieties. The following is the main theorem in \cite{DRS93}.

\begin{theorem}
(\cite{DRS93})
Suppose we have a semisimple $\mathbb{Q}-$group $G$ which admits a $\mathbb{Q}-$representation on a real vector space $W$. Take $\vec{w}\in W$ such that $G\vec{w}$ is Zariski closed in $W$. Suppose the stabilizer $H(\mathbb{R})$ of $\vec{w}$ is a `symmetric subgroup, the fixed elements of an involution of $G(\mathbb{R})$. In this case, $G\vec{w}$ is called a symmetric affine variety. Suppose that $Vol\left(H(\mathbb{R})/H(\mathbb{Z})\right)$ and $Vol\left(G(\mathbb{R})/G(\mathbb{Z})\right)$ both are finite. Furthermore, normalize $dg$ on $G(\mathbb{R})$ and $dh$ on $H(\mathbb{R})$ such that \[Vol\left(H(\mathbb{R})/H(\mathbb{Z})\right)=Vol\left(G(\mathbb{R})/G(\mathbb{Z})\right)=1.\] Moreover, there is a $G(\mathbb{R})-$invariant measure $d\dot{g}$ on $G(\mathbb{R})/H(\mathbb{R})$ satisfying $dg=d\dot{g}dh$. Define $\mu(T):=\int_{\norm{\dot{g}\vec{w}}\leq T}d\dot{g}$\ ($\norm{\cdot}$ means Euclidean norm here). Take $O=G(\mathbb{Z})\vec{w}$. Denote the distribution function of $O$ to be \[N(T,O)=\{\vec{w}'\in O|\norm{\vec{w}'}\leq T\}.\]Then we have the following asymptotic behavior:\[N(T,O)\sim \mu(T),  \text{ as }T\text{ tends to infinity}.\]
\end{theorem}

In the above theorem, the symmetric condition that $H$ consists of the fixed points of an involution plays a central role. In 1996 and 1997, \cite{EMS96} and \cite{EMS97} generalized the result to the cases that the stabilizer $H$ is not contained in any proper parabolic $\mathbb{Q}-$subgroup of $G$.

Let us focus on an example which we will further investigate in this paper. Consider $V_{n,k}:=\{A\in \mathrm{Mat}_{n\times n}(\mathbb{R})|\det(A)=k\}$. Take $G:=SL_n
\times SL_n$ with an action on $V_{n,k}$ given by $(g_1,g_2)x=g_1xg_2^{-1}$. Notice that $SL_n(\mathbb{R})/SL_n(\mathbb{Z})$ is of finite volume. Thus all the finiteness conditions hold. Over $\mathbb{R}$, the action is transitive on $V_{n,k}$. For any $x\in V_{n,k}$, the stabilizer is $H:=\{(xgx^{-1},g)|g\in SL_n(\mathbb{R})\}$. Consider the involution $\sigma(g_1,g_2)=(xg_2x^{-1},x^{-1}g_1x)$ on $G(\mathbb{R})$. Then $H$ consists of points fixed by $\sigma$. Thus we can take arbitrary $x\in V_{n,k}(\mathbb{Z})$ to be our $\vec{w}$ in the above theorem. By taking the union of all $SL_n(\mathbb{Z})\times SL_n(\mathbb{Z})-$orbits, we have the following asymptotic behavior as Example 1.6 in \cite{DRS93} shows:

\[N(T,V_{n,k})\sim c_{n,k}T^{n^2-n}, \text{as }T\text{ tends to infinity}\] where, if $k=\underset{1\leq i\leq r}{\Pi}p_i^{a_i}$, then \[c_{n,k}=\frac{\pi^{n^2/2}}{\Gamma(\frac{n^2-n+2}{2})\Gamma(\frac{n}{2})\underset{2\leq j\leq n}{\Pi}\zeta(j)}k^{1-n}\underset{1\leq j\leq r}{\prod}\underset{1\leq i\leq n-1}{\prod}\frac{p_j^{a_j+i}-1}{p_j^{i}-1}.\] 

In fact, the ratio between distributions of two $SL_n(\mathbb{Z})\times SL_n(\mathbb{Z})$-orbits is proportional to the ratio between the number of $SL_n(\mathbb{Z})-$orbits with respect to left (right) multiplication contained in them. Here is a concrete example:
\begin{example}
Take $n=2$ and $k=4$. There are only two $SL_n(\mathbb{Z})\times SL_n(\mathbb{Z})$-orbits, with representatives $\begin{pmatrix}2&0\\0&2\end{pmatrix}$ and $\begin{pmatrix}1&0\\0&4\end{pmatrix}$. We will denote the first one as $O_1$ and the latter one as $O_2$. We will find $N(T,O_1)/N(T,O_2)\to 1/6,  \text{as }T\text{ tends to infinity}$. 

On the other hand, there are seven $SL_n(\mathbb{Z})$-orbits with respect to left multiplication with representatives:
$\begin{pmatrix}2 &0\\0&2\end{pmatrix}\in O_1$ and $\begin{pmatrix}2 &1\\0&2\end{pmatrix}$,$\begin{pmatrix}1 &0\\0&4\end{pmatrix}$,$\begin{pmatrix}1 &1\\0&4\end{pmatrix}$, $\begin{pmatrix}1 &2\\0&4\end{pmatrix}$, $\begin{pmatrix}1 &3\\0&4\end{pmatrix}$, $\begin{pmatrix}4 &0\\0&1\end{pmatrix}\in O_2$.\end{example}

So there arises a natural question: 

\begin{question}
    For any two $SL_n(\mathbb{Z})$-orbit $O_1,O_2\subset V_{n,k}$, do we have the following asymptotic behavior:
\[N(T,O_1)\sim N(T,O_2), \text{ as }T\text{ tends to infinity?}\]
\end{question}
\begin{remark}
Notice that the theorem stated above does not apply to this case because the symmetric subgroup condition is not satisfied. The stabilizer here is a trivial group. And every automorphism of $SL_n(\mathbb{R})$ will preserve the center and 2-torsion elements. Thus every automorphism will preserve the set $\{I_n,-I_n\}$ which means that the stabilizer cannot be realized as fixed points of involution. \cite{EMS96} and \cite{EMS97} cannot be applied as well because the trivial group will be naturally contained in any proper parabolic $\mathbb{Q}-$subgroup.
\end{remark}

It's common to compare results over $\mathbb{Z}$ with results over $\mathbb{F}_q[x]$\ (here are some papers related to counting of points on homogeneous spaces over function fields \cite{FAILAI2002142} and \cite{THUNDER20082973}). To imitate the Euclidean norm, we are going to take the norm given by the valuation: $\deg$. 

Here are some reasons. These are the unique metrics such that $\mathbb{Z}$\ (resp. $\mathbb{F}_q[x]$) are discrete with respect to the corresponding completion. Thus they are the unique metrics for us to add restrictions to get a reasonable counting problem. 

After moving to $\mathbb{F}_q[x]$ in the rest, we will get a stronger result than what we conjecture above over $\mathbb{Z}$. The main reason is that $\deg$ is a non-archimedean valuation. We will also go back to $\mathbb{C}$ to give some geometric meaning of our results\ (Remark~\ref{geometric}).
\section{Main Results}

From now on, fix a finite field $\mathbb{F}_q$.
\begin{definition}
Given a matrix $M\in \mathrm{Mat}_{l\times n}(\mathbb{F}_q[x])$ and a natural number $k$, we say that $M$ satisfies property $\property{k}$, if all the entries of $M$ are of degree smaller or equal to $k$.
\end{definition}
Here is the equidistribution theorem we are going to prove:
\begin{theorem}\label{main1}
Suppose we are given an $M\in \mathrm{Mat}_{n\times n}(\mathbb{F}_q[x])$ such that $\det(M)\neq 0$ and $\deg(\det(M))=t$. Then for any $k\geq t$ we have 
\[\#\left\{g\in GL_n(\mathbb{F}_q[x])M:g\text{ satisfies }\property{k} \right\}=\#GL_n(\mathbb{F}_q) q^{(n-1)(nk-t)}=(q^{k})^{n^2-n} C_{n,t}\]
where $C_{n,t}=\#GL_n(\mathbb{F}_q)(q^{t})^{1-n}.$
\end{theorem}
\begin{remark}
    The conclusion of the theorem may fail if one omits the condition $k\geq t$. Suppose that $M$ is of the form $\mathrm{diag}(f,I_{n-1})$ where $\deg(f)=t$. The left hand side is zero when $k<t$ whereas the right hand side is never zero.
\end{remark}
\begin{corollary}
    Keeping all the assumptions in the above theorem, we have 
    \[\#\left\{g\in SL_n(\mathbb{F}_q[x])M:g\text{ satisfies }\property{k} \right\}=\#SL_n(\mathbb{F}_q) q^{(n-1)(nk-t)}.\]
\end{corollary}
\begin{proof}
Notice that $GL_n(\mathbb{F}_q)=SL_n(\mathbb{F}_q)\times \mathbb{F}_q^\times I_n$. So we know that 
\[\#GL_n(\mathbb{F}_q) q^{(n-1)(nk-t)}=(q-1)\#SL_n(\mathbb{F}_q) q^{(n-1)(nk-t)}.\]
Also, notice that for any $X\in GL_n(\mathbb{F}_q[x])M$ satisfying \property{k}, there exists a unique $a_X\in \mathbb{F}_q^\times$, which is the constant term of $\det(X)$ such that $a_X^{-1} X\in SL_n(\mathbb{F}_q[x])$. So we know that 
\[\#\left\{g\in GL_n(\mathbb{F}_q[x])M:g\text{ satisfies }\property{k} \right\}=(q-1)\#\left\{g\in SL_n(\mathbb{F}_q[x])M:g\text{ satisfies }\property{k} \right\}.\]
Combining the above two identities and the identity in the theorem, we will get the statement in the corollary.
\end{proof}
\begin{corollary}
For $k \geq t \geq 0$, we have
\begin{align*}
&\#\left\{g \in \mathrm{Mat}_{n \times n}(\mathbb{F}_q[x]) :  \deg(\det(g)) = t, \text{ and } g \text{ satisfies }\property{k} \right\} \\
=& (q^{k})^{n^2-n} p_{n,t} C_{n,t},
\end{align*}
where $p_{n,t} = \sum_{t_1 + t_2 + \cdots + t_n = t} q^{t_1 + 2t_2 + \cdots + nt_n}$.
\end{corollary}

\begin{remark}
The product $p_{n,t}C_{n,t}$ is analogous to the constant $c_{n,k}$.
\end{remark}

\begin{proof}
Let's introduce a lemma first.
\begin{lemma}\label{tech}
Suppose we are given $M \in \mathrm{Mat}_{n \times n}(\mathbb{F}_q[x])$ where $\det(M) \neq 0$. Then we can find a unique upper-triangular matrix $B = (b_{i,j})_{1 \leq i,j \leq n} \in GL_n(\mathbb{F}_q[x])M$ such that $\deg(b_{i,j}) < \deg(b_{j,j})$ where $i < j$ and $b_{i,i}, 1 \leq i \leq n$ are monic polynomials.
\end{lemma}

By this lemma, we have the following identity
\begin{align*}
&\#\left\{g \in \mathrm{Mat}_{n \times n}(\mathbb{F}_q[x]) :  \deg(\det(g)) = t, \text{ and } g \text{ satisfies }\property{k} \right\} \\
=& \sum_{B \in S_t} \#\left\{g \in GL_n(\mathbb{F}_q[x])B :  g \text{ satisfies }\property{k} \right\},
\end{align*}
where $S_t$ is the set of upper-triangular matrices $B = (b_{i,j})_{1 \leq i,j \leq n}$ such that $\deg(\det(B)) = t$, $\deg(b_{i,j}) < \deg(b_{j,j})$ where $i < j$, and $b_{i,i}, 1 \leq i \leq n$ are monic polynomials. By Theorem~\ref{main1}, for every $B \in S_t$, we have \[\#\left\{g \in GL_n(\mathbb{F}_q[x])B :  g \text{ satisfies }\property{k} \right\} = (q^k)^{n^2-n}C_{n,t}.\] By the degree restrictions, we can easily count $\#S_T$ and find that it's $p_{n,t} = \sum_{t_1 + t_2 + \cdots + t_n = t} q^{t_1 + 2t_2 + \cdots + nt_n}$. Combining the above identities, we get the formula stated in the Corollary.

Now let's start the proof of Lemma~\ref{tech}.
\begin{proof}
Let's prove it by induction on $n$. 

When $n = 1$, the argument automatically holds. 

Suppose that $n \geq 2$. The uniqueness can be interpreted as follows: Assume that we have $B, gB \in GL_n(\mathbb{F}_q[x])M, g \in GL_n(\mathbb{F}_q[x])$ satisfying the properties. Then $g = I_n$.

Let's write $g$ blockwise as $\begin{pmatrix}g_1 & g_2 & g_3\\0 & g_4 & g_5\\0 & 0 & g_6\end{pmatrix}$, where $g_1, g_3, g_6 \in \mathbb{F}_q[x]$, $g_2^T, g_5 \in \mathbb{F}_q[x]^{n-2}$, and $g_4 \in GL_n(\mathbb{F}_q[x])$. We can use the induction hypothesis on the upper left $(n-1) \times (n-1)$ part and the downward right $(n-1) \times (n-1)$ part. Then we know that $g_1 = g_6 = 1$, $g_2^T = g_5 = \vec{0}$, and $g_4 = I_{n-2}$. Now it suffices to show that $g_3 = 0$. Notice that the $(1,n)$-entry (resp. $(n,n)$-entry) of $gB$ is now $b_{1,n} + g_3b_{n,n}$ (resp. $b_{n,n}$). By the degree restrictions, we know that $\deg(b_{1,n} + g_3b_{n,n}) < \deg(b_{n,n})$ and $\deg(b_{1,n}) < \deg(b_{n,n})$. Using the strong triangular inequality, we know that $\deg(g_3b_{n,n}) < \deg(b_{n,n})$, which implies $g_3 = 0$. 

Now let's focus on the existence. Since we are working over $\mathbb{F}_q[x]$, a principal ideal domain, we are able to find an element $p \in GL_n(\mathbb{F}_q[x])$ such that the first column of $pM$ only has nonzero entries in the first row. Denote the downward-right $(n-1) \times (n-1)$ part of $pM$ by $\tilde{M}$. We may use the induction hypothesis to find a $q \in GL_{n-1}(\mathbb{F}_q[x])$ such that $q\tilde{M}$ is upper-triangular. So we know that $\mathrm{diag}(1,q)pM \in GL_n(\mathbb{F}_q[x]M$ is upper-triangular. Now we may assume that our $M$ is upper-triangular. We only need to find $B \in GL_n(\mathbb{F}_q[x])M$ satisfying the degree restrictions. Interested readers may find a proof using left multiplication by elementary matrices step by step. We produce a simplified proof based on the induction hypothesis here. Let's write $B$ blockwise as $\begin{pmatrix}B_1 & B_2 & B_3\\0 & B_4 & B_5\\0 & 0 & B_6\end{pmatrix}$, where $B_1, B_3, B_6 \in \mathbb{F}_q[x]$, $B_2^T, B_5 \in \mathbb{F}_q[x]^{n-2}$, and $B_4 \in \mathrm{Mat}_{(n-2) \times (n-2)}(\mathbb{F}_q[x])$. Using induction hypothesis on the upper left $(n-1) \times (n-1)$ part first and then on the downward right $(n-1) \times (n-1)$ part (the order matters here), we may assume that only one of the degree restrictions $\deg(B_{1,n}) < \deg(B_{n,n})$ has not been satisfied. Since we require that $\det(M) \neq 0$, we know that $B_{n,n} \neq 0$. By the Euclidean division algorithm, we may find an element $f \in \mathbb{F}_{q}[x]$ such that $\deg(B_{1,n} + fB_{n,n}) < \deg(B_{n,n})$. Denote the elementary matrix representing adding the $n$th row multiplied by $f$ to the first row by $L(f)$. Then $L(f)B$ differs from $B$ only on the $(1,n)$th entry where the $(1,n)$th entry of $L(f)B$ is of degree strictly smaller than $\deg(B_{n,n})$. So we know that $L(k)B \in GL_n(\mathbb{F}_q[x])M$ gives us the existence. 
\end{proof}
\end{proof}

We are going to break the proof of Theorem 2.2 into two parts.

Notice that, by Lemma~\ref{tech}, we can always find an upper-triangular matrix $B\in GL_n(\mathbb{F}_q[x])M$. Now we may assume $M$ is upper-triangular. In fact, we are able to assume that $M$ is diagonal by taking $l=n$ in the following lemma.
\begin{lemma}\label{lemma1}
Suppose we are given two positive integers $1\leq l\leq n$ and an upper-triangular matrix $M=\begin{pmatrix}D & *\\ 0 &*\end{pmatrix}\in \mathrm{Mat}_{n\times n}(\mathbb{F}_q[x]),D\in \mathrm{Mat}_{l\times l}(\mathbb{F}_q[x])$.  Then,  there exists an upper-triangular $M'=\begin{pmatrix}D' & *\\ 0 &*\end{pmatrix}\in \mathrm{Mat}_{n\times n}(\mathbb{F}_q[x])$ such that $D'\in \mathrm{Mat}_{l\times l}(\mathbb{F}_q[x])$ is a diagonal matrix, $\deg(\det(D'))\geq \deg(\det(D))$, $\deg(\det(M'))=\deg(\det(M))$ and for any $k\geq 0$, we have
\[\#\{g\in GL_{n}(\mathbb{F}_q[x]):gM\text{ satisfies } \property{k}\}=\#\{g\in GL_{n}(\mathbb{F}_q[x]):gM'\text{ satisfies } \property{k}\}.\]
\end{lemma}
\begin{remark}
We only require an identity on the size of the sets instead of an identity of the sets. The equality of the sets does not hold in general.
\end{remark}
From now on, we may assume $M$ is a diagonal matrix. Notice that when $M$ is a diagonal matrix, we only care for the degree of its entries on the diagonal. So we are reduced to calculating the size of the following sets:

\[\left\{A=(a_{i,j})\in GL_n(\mathbb{F}_q[x]):\deg(a_{i,j})\leq k-t_j,\forall  1\leq i,j\leq n \right\}\]
where $t_i$ is the degree of the $(i,i)$th entry of the diagonal matrix and $t_1+t_2+\cdots+t_n=t$.

For any $A\in \mathrm{Mat}_{n\times n}(\mathbb{F}_q[x])$, we may write it as $A=\sum_{d\geq 0} A_d x^d$ where $A_d\in \mathrm{Mat}_{n\times n}(\mathbb{F}_q)$. If $A$ is invertible, we know that $A_0\in GL_n(\mathbb{F}_q)$ is invertible. Notice that left multiplication by $GL_n(\mathbb{F}_q)$ stabilizes the set which we are counting and leaves $\property{k}$ invariant. 
So it suffices to count those with $A_0=I_n$.
We utilize the following lemma:
\begin{lemma}\label{lemma2}
For all $k_j\geq 0,1\leq j\leq n$, we have
\begin{equation}\label{lemma2id}\#\left\{A=(a_{i,j})\in GL_n(\mathbb{F}_q[x]):\deg(a_{i,j})\leq k_j,\forall  1\leq i,j\leq n;A_0=I_n \right\}=q^{(n-1)\underset{1\leq i\leq n}{\sum}k_i}.\end{equation}
\end{lemma}
\begin{remark}\label{geometric}
In fact, the set on the left side of (\ref{lemma2id}) can be viewed as the $\mathbb{F}_q$-points of a scheme defined over $\mathbb{Z}$. This lemma tells us this scheme is strong polynomial count over $\mathbb{Z}$ (for definition, see the appendix of \cite{HRV08}).  In an answer (\cite{onlineAns} ) posted on mathoverflow by Will Sawin, he proved that the \'{e}tale cohomology of any closed, scaling-invariant subset of $\mathbb{A}^n$ containing the origin is $\mathbb{Q}_l$ in degree $0$. To establish (\ref{lemma2id}) using the Grothendieck trace formula , it is essential to demonstrate that the variety we defined is rationally smooth of dimension $(n-1)\underset{1\leq i\leq n}{\sum}k_i$. When we take $k_i=1,\forall 1\leq i\leq n$, the left side actually corresponds to nilpotent cone which was proved rationally smooth in \cite{BM83}. One could try and pursue Lemma~\ref{lemma2} using algebraic geometry, but we only give a completely elementary proof here.
\end{remark}

The organization of this paper is as follows. We shall prove Lemma~\ref{lemma1} in Section~\ref{section1}. Then we will prove Lemma~\ref{lemma2} in Section~\ref{section2}.

\section{Proof of Lemma~\ref{lemma1}}\label{section1}
We are going to prove this lemma by induction on $l$.

Notice that the base case $l=1$ is trivial, since every upper-triangular matrix $M$ satisfies the condition itself.

Now let's deal with the induction part. Suppose the lemma is proved in the case that $l=l_0-1,l_0\geq 2$. We are now attacking the case that $l=l_0$.

Suppose the original upper-triangular matrix is \[M_0=\begin{pmatrix}D_0&*\\ 0 &*\end{pmatrix}\in \mathrm{Mat}_{n\times n}(\mathbb{F}_q[x]),\det(M_0)\neq 0, D_0\in Mat_{(l_0-1)\times (l_0-1)}(\mathbb{F}_q[x]).\]
Suppose our induction step for $M_0$ fails.

Firstly, we are going to prove the following technical lemma:
\begin{lemma}
For the above $M_0$, the following set:
\[S_{M_0}:=\left\{M\in \mathrm{Mat}_{n\times n}(\mathbb{F}_q[x]):M=(m_{i,j})_{1\leq i,j\leq n}\text{ satisfies the following conditions}\right\}\] is nonempty where the conditions are:
\begin{equation}\label{equation1}\begin{split}&M=\begin{pmatrix}D & * \\0&*\end{pmatrix}, D=\mathrm{diag}(d_1,d_2,\ldots,d_{l_0-1}) \\ &\deg(\det(D))\geq \deg(\det(D_0)),\deg(\det(M))=\deg(\det(M_0))\end{split}\end{equation}
\begin{equation}\label{condition0}
    \#\{g\in GL_{n}(\mathbb{F}_q[x]):gM_0\text{ satisfies } \property{k}\}=\#\{g\in GL_{n}(\mathbb{F}_q[x]):gM\text{ satisfies } \property{k}\}
\end{equation}
\begin{equation}\label{equation2}
\deg(d_1)\leq \deg(d_2)\leq\cdots \leq \deg(d_{l_0-1})\end{equation}

\begin{equation}\label{equation3}
x^{\deg(d_i)+1}\text{ divides the }(i,l_0)\text{th entry of }M,\forall 1\leq i\leq l_0-1
\end{equation}
\begin{equation}\label{identity1}\deg(m_{i,l_0})+1\leq \deg(m_{l_0,l_0}),\forall 1\leq i<l_0 \end{equation}
\begin{equation}\label{assumption0}
    x^{\deg(d_1)+1}\text{ divides the }l_0\text{th column of }M
\end{equation}
\end{lemma}
\begin{proof}
By the induction hypothesis, we may take an upper-triangular matrix $M=(M_{i,j})_{1\leq i,j\leq n}$ which satisfies conditions (\ref{equation1}) and (\ref{condition0}).

Notice that conjugation by a matrix in $GL_n(\mathbb{F}_q)$ does not affect the property $\property{k}$ for any $k$. Thus, for any $h\in GL_n(\mathbb{F}_q)$ we have 
\begin{align*}&\#\{g\in GL_{n}(\mathbb{F}_q[x]):gM\text{ satisfies } \property{k}\}\\=&\#\{g\in GL_{n}(\mathbb{F}_q[x]):hgMh^{-1}\text{ satisfies } \property{k}\}\\=&\#\{g\in GL_{n}(\mathbb{F}_q[x]):hg h^{-1}hMh^{-1}\text{ satisfies } \property{k}\}\\=&\#\{g\in GL_{n}(\mathbb{F}_q[x]):g hMh^{-1}\text{ satisfies } \property{k}\}.\end{align*}

So we may choose $h$ to be a permutation matrix which only permutes the elements of $D$ to further assume that $M$ also satisfies the condition (\ref{equation2})\ (notice that this step does not change the condition~(\ref{equation1}) and (\ref{condition0})).

For a polynomial $f=\underset{i\geq 0}{\sum} a_i x^i\in \mathbb{F}_q[x]$ and $u>0$, we define \[T_u(f):=\underset{0\leq i\leq u}{\sum} a_i x^i.\]

To produce a matrix satisfying the remaining conditions, take $M'=(m'_{i,j})_{1\leq i,j\leq n}\in \mathrm{Mat}_{n\times n}(\mathbb{F}_q[x])$ such that
\[m'_{i,l_0}=m_{i,l_0}-T_{\deg(d_i)}(m_{i,l_0}),1\leq i\leq l_0-1;\]\[m'_{i,j}=m_{i,j},\text{ otherwise}.\]

Notice that $M'$ is still an upper-triangular matrix satisfying conditions~(\ref{equation1}) and (\ref{equation2}). And, by the construction,  $M'$ satisfies the condition (\ref{equation3}).

We are going to prove that
\begin{equation}\label{condition3}\{g\in GL_{n}(\mathbb{F}_q[x]):gM\text{ satisfies } \property{k}\}=\{g\in GL_{n}(\mathbb{F}_q[x]):gM'\text{ satisfies } \property{k}\}\end{equation}

This will imply that $M'$ also satisfies condition~(\ref{condition0}).

Denote the $i$th column vector of $M$\ (resp. $M'$) by $\alpha_i$\ (resp. $\beta_i$). Notice that we only change the $l_0$th column. 
So we will only change the $l_0$th column of $gM$. 
Thus if the identity$~(\ref{condition3})$ fails, there exists a vector
$\gamma\in \mathbb{F}_q[x]^n$\ (where $\gamma^T$
is some row vector of an element $g\in GL_n(\mathbb{F}_q[x])$) 
such that either:
\[\deg(\gamma^T\alpha_{l_0})\leq k,\deg(\gamma^T\beta_{l_0})>k\text{ and }\deg(\gamma^T\alpha_{i}) \leq k,\forall i<l_0; \]
or
\[\deg(\gamma^T\alpha_{l_0})> k,\deg(\gamma^T\beta_{l_0})\leq k\text{ and }\deg(\gamma^T\alpha_{i}) \leq k,\forall i<l_0.\]

In both cases, we have 
\[\deg(\gamma^T(\alpha_{l_0}-\beta_{l_0}))>k\text{ and }\deg(\gamma^T\alpha_{i}) \leq k,\forall i<l_0\]

Suppose the first $l_0-1$ entries of $\gamma$ are $f_1,\ldots, f_{l_0-1}$. Expanding the above inequalities in detail, we get that 
\[
\deg(\underset{1\leq i\leq l_0-1}{\sum}f_iT_{\deg(d_i)}(m_{i,l_0}))>k ;
\]
\[\deg(f_i)+\deg(d_i) \leq k, \forall 1\leq i\leq l_0-1.\]

By the latter one and the non-archimedean property of $\deg$, we have:
\begin{align*}
&\deg(\underset{1\leq i\leq l_0-1}{\sum}f_iT_{\deg(d_i)}(m_{i,l_0}))\\\leq& \max\{\deg(f_iT_{\deg(d_i)}(m_{i,l_0})), 1\leq i\leq l_0-1\}\\\leq&\max\{\deg(f_i)+\deg(d_i),1\leq i\leq l_0-1\}\\\leq & k
\end{align*}
which contradicts with the first inequality. In all, we have proved that the identity~(\ref{condition3}) holds.

So far, we have shown that $M'$ satisfies conditions (\ref{equation1}),(\ref{condition0}),(\ref{equation2}) and (\ref{equation3}).

Notice that we may use $GL_n(\mathbb{F}_q[x])$ to reduce the entries above $m'_{l_0,l_0}$ to elements of degree strictly smaller than $\deg(m'_{l_0,l_0})$ without changing the diagonal matrix $D$ attached to $M'$. This reduction also won't change conditions listed in $S_{M_0}$ which $M'$ already satisfies.  Hence we may further assume that $M'$ also satisfies the condition (\ref{identity1}).

Now if $\deg(m'_{l_0,l_0})\leq \deg(d_1)+1$, by the conditions (\ref{equation2}), (\ref{equation3}) and (\ref{identity1}), we know that $M'$ is already of the following form:
\[\begin{pmatrix}D'&*\\0&*\end{pmatrix},D'=
\mathrm{diag}(d_1,d_2,\ldots,d_{l_0-1},d_{l_0})\] for some $d_{l_0}\in \mathbb{F}_q[x]$. Hence the induction step works for $M_0$ which contradicts the assumption. So we can assume that $\deg(m'_{l_0,l_0})> \deg(d_1)+1$.

We are going to construct $M''=(m''_{i,j})_{1\leq i,j\leq n}$ as follows:
\[m''_{l_0,l_0}:= m'_{l_0,l_0}-T_{\deg(d_1)+1}(m'_{l_0,l_0});\]
\[m''_{i,j}=m'_{i,j},\text{ otherwise}.\]
Notice that by the assumption above and conditions (\ref{equation2}), (\ref{equation3}), the $M''$ which we just constructed satisfies condition (\ref{assumption0}).

We are going to prove that $M''$ satisfies 
\begin{equation}\label{condition4}\{g\in GL_{n}(\mathbb{F}_q[x]):gM'\text{ satisfies } \property{k}\}=\{g\in GL_{n}(\mathbb{F}_q[x]):gM''\text{ satisfies } \property{k}\}\end{equation}
which implies that $M''$ satisfies condition~(\ref{condition0}).

Suppose the identity~(\ref{condition4}) does not hold. Imitating what we have done for $M'$, we know that there exists a row vector $(f_1,f_2,\cdots,f_n)\in \mathbb{F}_q[x]^n$ such that either
\[\deg(\underset{1\leq i\leq l_0}{\sum}f_i m'_{i,l_0})>k,\deg(\underset{1\leq i\leq l_0}{\sum}f_i m''_{i,l_0})\leq k\text{ and }\deg(f_im'_{i,i})\leq k, \forall  1\leq i<l_0;\]
or
\[\deg(\underset{1\leq i\leq l_0}{\sum}f_i m'_{i,l_0})\leq k,\deg(\underset{1\leq i\leq l_0}{\sum}f_i m''_{i,l_0})>k\text{ and }\deg(f_im'_{i,i})\leq k, \forall 1\leq i<l_0.\]
holds. In both cases, we have
\[\deg(f_{l_0} T_{\deg(d_1)+1}(m'_{l_0,l_0}))>k\text{ and }\deg(f_im'_{i,i})\leq k, \forall 1\leq i<l_0.\]
So we know that
\begin{equation}\label{identity2}
    \deg(f_{l_0})>k-1-\deg(d_1);
\end{equation}
\begin{equation}\label{identity14}
    \deg(f_{l_0}m''_{l_0,l_0})=\deg(f_{l_0}m'_{l_0,l_0})>k+\deg(m'_{l_0,l_0})-\deg(d_1)-1>k.
\end{equation}

By the fact that one of $\underset{1\leq i\leq l_0}{\sum}f_i m''_{i,l_0}, \underset{1\leq i\leq l_0}{\sum}f_i m'_{i,l_0} $ is of degree smaller of equal than $k$ and the identity~(\ref{identity14}), we know there exists $1\leq i< l_0$ such that \begin{equation}\label{identity0} \deg(f_i m'_{i,l_0})=\deg(f_{l_0}m''_{l_0,l_0}).\end{equation} 

Combining identities and inequalities listed above, we have the following result:
\begin{align*}
    \deg(f_im'_{i,i})
    &= \deg(f_i) + \deg(d_i) \\
    &\overset{(\ref{identity0})}{=} \deg(f_{l_0}) + \deg(m''_{l_0,l_0}) - \deg(m'_{i,l_0}) + \deg(d_i) \\
    &\overset{(\ref{identity2})}{>} k-1 - \deg(d_1) + \deg(m''_{l_0,l_0}) - \deg(m'_{i,l_0}) + \deg(d_i) \\
    &\geq k-1 + \deg(m''_{l_0,l_0}) - \deg(m'_{i,l_0}) \\
    &\overset{(\ref{identity1})}{\geq} k
\end{align*}

which contradicts with the inequality $\deg(f_i m'_{i,i})\leq k$. Hence we have proved that $M''$ satisfies (\ref{condition0}). Notice that the construction of $M''$ does not affect other conditions in $S_{M_0}$ which $M'$ already satisfies. 

Above all, we have constructed an $M''\in S_{M_0}$.
\end{proof}
Notice that for any element $M\in S_{M_0}$ and its attached $d_1,d_2,...,d_{l_0-1}$, we have \[\deg(d_1)+\deg(d_2)+\cdots+\deg(d_{l_0-1})\leq \deg(\det(M))=\deg(\det(M_0)).\]
Thus we may pick $M_{max}\in S_{M_0}$ such that $\deg(d_1)+\deg(d_2)+\cdots+\deg(d_{l_0-1})$ is maximal.

Let's conjugate $M_{max}$ by the permutation matrix corresponding to the two-cycle $(1\ l_0)\in S_n$ to get a matrix $M'_{max}$.

By what we formulated in the proof of the foregoing lemma, $M'_{max}$ still satisfies the condition (\ref{condition0}). Then we may find an upper-triangular matrix in the orbit $GL_n(\mathbb{F}_q[x])M'_{max}$: \[N=\begin{pmatrix}D_N&*\\0&*\end{pmatrix}, D_N\in Mat_{(l_0-1)\times (l_0-1)}(\mathbb{F}_q[x]).\] 
This $N$ also satisfies the condition (\ref{condition0}).

Notice that the above operation does not change the g.c.d of the columns. Thus, using (\ref{equation1}) and (\ref{assumption0}), we know that the g.c.d. of the $i$th column of $N$ is of degree:
\[\begin{cases}\geq \deg(d_1)+1,&i=1\\= \deg(d_i),&2\leq i\leq l_0-1\\=\deg(d_1),& i=l_0
\end{cases}.\]
Thus we know that \[\deg(\det(D_N))\:\geq\: \deg(d_1)+\deg(d_2)+\cdots+\deg(d_{l_0-1})+1\:>\:\deg(\det(D_0)).\]

As well, notice that $\det(N)=c\det(M_{max})$ for some $x\in \mathbb{F}_q^{\times}$. Combining this with the fact that $M_{max}$ satisfies the condition (\ref{equation1}), we know that $\deg(\det(N))=\deg(\det(M_0))$.

Above all, we see that $S_N\subset S_{M_0}$. Also, if the induction step works for $N$, it will work for $M_0$. So, by our assumption, the induction step for $N$ fails as well. Then we may apply the above lemma to $N$, i.e. $S_N$ is not empty. 

Pick up one element $N'\in S_N$ with attached $d'_1,d'_2,\ldots,d'_{l_0-1}$. By the condition (\ref{equation1}), we see that $N'\in S_{M_0}$ has strictly bigger $\deg(d'_1)+\deg(d'_2)+\cdots+\deg(d'_{l_0-1})$ compared with $M_{max}$, which contradicts with our choice of $M_{max}$. Hence we have shown that the induction step for $M_0$ works.

\section{Proof of Lemma~\ref{lemma2}\label{section2}}
We shall prove Lemma~\ref{lemma2} by induction both on $n$ and the sum $\underset{1\leq i\leq n}{\sum}k_i$.

The base case for $n=1$ and $\underset{1\leq i\leq n}{\sum}k_i=0$ is trivial. It suffices to prove the induction step.

Pick $n_0,k> 0$. Suppose that Lemma~\ref{lemma2} holds for any $n,k_1,\ldots,k_n$ such that $n< n_0$ or $\underset{1\leq i\leq n}{\sum}k_i< k$. Now we attack the case where $n=n_0$ or $\underset{1\leq i\leq n}{\sum}k_i=k$.

Define $P_{n;k_1,k_2,\ldots,k_n}$ to be
\[P_{n;k_1,k_2,\ldots,k_n}:=\{A=(a_{i,j})\in GL_n(\mathbb{F}_q[x]):\deg(a_{i,j})\leq k_j,\forall 1\leq i\leq n;A_0=I_n \}.\]

For any $P_{n;k_1,k_2,...,k_{n}} $, let's rewrite $A$ as $\begin{pmatrix}\alpha_1&\alpha_2&\cdots&\alpha_n \end{pmatrix}$ where $\alpha_i$ are column vectors in $\mathbb{F}_q[x]$. Then we decompose these vectors in each degree as follows:
\[\alpha_i=\underset{0\leq j\leq k_i}{\sum}  \alpha_i^{(j)} x^j, \alpha_i^{(j)}\in \mathbb{F}_q^n, i=1,2,\cdots,n_0.\]

Notice that for any $A\in P_{n_0;k_1,k_2,\ldots,k_{n_0}}$, we have $\det(A)\in \mathbb{F}_q$. By checking the degree $k_1+k_2+\cdots+k_{n_0}$ part, we know that  $\det(\begin{pmatrix}\alpha^{(k_1)}_1,\alpha^{(k_2)}_2,\cdots,\alpha^{(k_{n_0})}_{n_0}\end{pmatrix})=0$ which means $\alpha^{(k_1)}_1,\alpha^{(k_2)}_2,\ldots,\alpha^{(k_{n_0})}_{n_0}$ are linearly dependent. 

Let's introduce one further definition before going back to the main proof:
\[
Q^i_{n_0;l_1,\ldots,l_{n_0}}:=\left\{A\in P_{n_0;l_1,l_2,\ldots,l_{n_0}}:\substack{ \alpha^{(l_i)}_i,\ldots,\alpha^{(l_{n_0})}_{n_0}\text{ are linearly dependent}\\ \alpha^{(l_{i+1})}_{i+1},\ldots,\alpha^{(l_{n_0})}_{n_0}\text{ are linearly independent}}\right\}, i=1,2,\ldots, n_0
\]
\[R^i_{n_0;l_1,\ldots,l_{n_0}}:=\left\{A\in P_{n_0;l_1,l_2,\ldots,l_{n_0}}:  \alpha^{(l_{i})}_{i},\ldots,\alpha^{(l_{n_0})}_{n_0}\text{ are linearly dependent}\right\}, i=1,2,\ldots, n_0\]
where \begin{equation}\label{counting0}Q^{n_0}_{n_0;l_1,\ldots,l_{n_0}}=P_{n_0;l_1,l_2,\ldots,l_{n_0}-1}=R^{n_0}_{n_0;l_1,\ldots,l_{n_0}}.\end{equation}

Moreover, we have \begin{equation}\label{counting1} R^i_{n_0;l_1,\ldots,l_{n_0}}=\underset{n_0\geq j\geq i} {\bigsqcup} Q^j_{n_0;l_1,\ldots,l_{n_0}}. \end{equation}

With these notations, let's go back to the main proof. 

Firstly, notice that we may assume $k_1\geq k_2\geq \cdots \geq k_{n_0}$.

If $k_{n_0}=0$, the number which we want to count is:
\[\#\{A=\begin{pmatrix}A'&0\\ *&1\end{pmatrix}\in GL_n(\mathbb{F}_q[x]):A'\in P_{n_0-1;k_1,k_2,\ldots,k_{n_0-1}}\}.\] The $*$ part does not affect the invertibility because $\begin{pmatrix}A'&0\\B&1\end{pmatrix}^{-1}=\begin{pmatrix}A'^{-1}&0\\-BA'^{-1}&1\end{pmatrix}$.  Thus we know that the number we are exactly counting is 
\[q^{k_1+k_2+\cdots+k_{n_0-1}}\#P_{n_0-1;k_1,k_2,\ldots,k_{n_0-1}}.\]
By the induction hypothesis, we will get the result we want.

Now it suffices to deal with the case that $k_{n_0}\geq 1$.
\begin{lemma}\label{useful}
Given a descending sequence $l_i\geq l_{i+1}\geq \cdots\geq l_{n_0}\geq 1, 2\leq i\leq n_0$ and  any nonnegative sequence $l_1,\ldots,l_{i-1}$, we have 
\begin{equation}\label{counting3}
    \#R^i_{n_0;l_1,\cdots,l_{n_0}}=\underset{\tilde{l}=(\tilde{l}_{i},\tilde{l}_{i+1},\cdots,\tilde{l}_{n_0})\in L}{\sum} c_{\tilde{l}}\#P_{l_1,l_2,\cdots,l_{i-1},\tilde{l}_i,\tilde{l}_{i+1},\cdots,\tilde{l}_{n_0}}
\end{equation}
where $L$ is a set of $n_0-i+1$ tuples which only depends on $i,l_{i},l_{i+1},\cdots,l_{n_0}$ such that for any tuple $\tilde{l}\in L$, we have 
\[\tilde{l}_{i}+\tilde{l}_{i+1}+\cdots+\tilde{l}_{n_0}<l_i+l_{i+1}+\cdots+l_{n_0}\]
and $c_{\tilde{l}}$ only depends on $i$ and $\tilde{l}$.
\end{lemma}
\begin{proof}

Let's prove this lemma by induction on $n_0-i$. 

When $n_0-i=0$, we know that $R^i_{n_0;l_1,\cdots,l_{n_0}}$ is exactly $P_{n_0;l_1,\cdots,l_{n_0}-1}$ since a single vector is linearly dependent if and only if it's zero. Thus we may take $c_{\tilde{l}}$ to be $1$ when $\tilde{l}=(l_{n_0}-1)$ and to be $0$ for other choices to make the formula works. Thus the base case is proved.

Now let's attack the induction step.

Notice that we can use the last $n_0-i$ columns to reduce elements in $Q^i_{n_0;l_1,\cdots,l_{n_0}}$ to elements in $P_{n_0;l_1,\cdots,l_{i-1},l_i-1,l_{i+1},\cdots,l_{n_0}}\setminus R^{i+1}_{n_0;l_1,\cdots,l_{i-1},l_i-1,l_{i+1},\cdots,l_{n_0}} $\ (due to the ascending assumption). Let's call this map $f$. 

And every element $A$ in $P_{n_0;l_1,\ldots,l_{i-1},l_i-1,l_{i+1},\ldots,l_{n_0}}\setminus R^{i+1}_{n_0;l_1,\ldots,l_{i-1},l_i-1,l_{i+1},\ldots,l_{n_0}} $ give us exactly $q^{n_0-i}$ different elements in $Q^i_{n_0;l_1,\ldots,l_{n_0}}$ which is the preimage $f^{-1}(A)$. 

Thus we know that:
\begin{equation}\label{counting2}\#Q^i_{n_0;l_1,\ldots,l_{n_0}}=q^{n_0-i}(\#P_{n_0;l_1,\ldots,l_{i-1},l_i-1,l_{i+1},\ldots,l_{n_0}}-\#R^{i+1}_{n_0;l_1,\ldots,l_{i-1},l_i-1,l_{i+1},\ldots,l_{n_0}}).\end{equation}

For the new sub-index of $R$, the last $n_0-i$ elements still form a decreasing sequence. So once we substitute identity~(\ref{counting1}) into (\ref{counting2}), we are able to use the induction hypothesis to prove the induction step.
\end{proof}
\begin{corollary}
In addition to those conditions in Lemma~\ref{useful}, we further assume that $l_{1}\geq 1$ and $\underset{1\leq i\leq n_0}{\sum}l_i\leq k$. Then we have 
\begin{equation}\label{counting4}\#R^i_{n_0;l_1,\cdots,l_{n_0}}=q^{n_0-1}\#R^i_{n_0;l_1-1,\cdots,l_{n_0}}.\end{equation}
\end{corollary}
\begin{proof}
Just use the induction hypothesis directly in the identity~(\ref{counting3}):
\begin{equation}
    \begin{split}&\#R^i_{n_0;l_1,\cdots,l_{n_0}}\\=&\underset{\tilde{l}=(\tilde{l}_{i},\tilde{l}_{i},\cdots,\tilde{l}_{n_0})\in L}{\sum} c_{\tilde{l}}\#P_{l_1,l_2,\cdots,l_{i-1},\tilde{l}_i,\tilde{l}_{i+1},\cdots,\tilde{l}_{n_0}}\\=&q^{n_0-1}\underset{\tilde{l}=(\tilde{l}_{i},\tilde{l}_{i},\cdots,\tilde{l}_{n_0})\in L}{\sum} c_{\tilde{l}}\#P_{l_1-1,l_2,\cdots,l_{i-1},\tilde{l}_i,\tilde{l}_{i+1},\cdots,\tilde{l}_{n_0}}
    \\=&q^{n_0-1}\#R^i_{n_0;l_1-1,\cdots,l_{n_0}}.\end{split}
\end{equation}
\end{proof}

Let's take $i=2$ and $l_j=k_j,1\leq j\leq n_0$ into the identity~(\ref{counting4}), we will get that
\begin{equation}\label{final}\#R^2_{n_0;k_1,k_2,\cdots,k_{n_0}}=q^{n_0-1}\#R^2_{n_0;k_1-1,k_2,\cdots,k_{n_0}}.\end{equation}

Above all, we are able to make the following calculation:
\begin{align*}
    &\#P_{n_0;k_1,k_2,\cdots,k_{n_0}}\\=&\#Q^1_{n_0;k_1,k_2,\cdots,k_{n_0}}+\#R^2_{n_0;k_1,k_2,\cdots,k_{n_0}}\\\overset{(\ref{counting2})}{=}&q^{n_0-1}(\#P_{n_0;k_1-1,k_2,\cdots,k_{n_0}}-\#R^2_{n_0;k_1-1,k_2,\cdots,k_{n_0}})+\#R^2_{n_0;k_1,k_2,\cdots,k_{n_0}}\\\overset{(\ref{final})}{=}&q^{n_0-1}\#P_{n_0;k_1-1,k_2,\cdots,k_{n_0}}\\=&q^{(n_0-1)\underset{1\leq i\leq n_0}{\sum}k_i}.
\end{align*}
The last identity is due to the induction hypothesis. We are done with the induction step now.
\section*{Acknowledgement}

The author is grateful to Jordan Ellenberg, Allen Knutson, Liang Xiao, and Runlin Zhang for engaging in insightful discussions. The author would also like to extend their appreciation to the reviewer for their invaluable comments.

\bibliographystyle{abbrv}
\bibliography{referrences}

\end{document}